\documentclass[12pt]{amsart}
\usepackage [T1]{fontenc}
\usepackage [latin1]{inputenc}
\usepackage{amssymb,amsmath,amsthm}
\usepackage{verbatim}
\usepackage{graphicx}
\usepackage{epsfig}

\usepackage{color, soul}
\usepackage[all]{xy}

\definecolor{dgreen}{cmyk}{1,.5,1,.2}

\newtheorem{theorem}{Theorem}[section]
\newtheorem{definition}[theorem]{Definition}

\newtheorem{lemma}[theorem]{Lemma}
\newtheorem{proposition}[theorem]{Proposition}

\newtheorem{remark}[theorem]{Remark}

\def\s{\backslash}
\def\os{\otimes^{n,s}}

\def\L{\mathcal{L}}
\def\P{\mathcal{P}^n}
\def\p{\mathcal{P}^n}
\def\Q{\mathcal{Q}}
\def\q{\mathcal{Q}}
\def\U{\mathfrak{A}}
\begin{document}

\title[Natural symmetric tensor norms]{Natural symmetric tensor norms}

\thanks{The first author was partially supported by ANPCyT PICT 05 17-33042, UBACyT Grant X038 and ANPCyT PICT 06 00587. The second author was partially supported by ANPCyT PICT 05 17-33042, UBACyT Grant X863 and a Doctoral fellowship from CONICET}

\author{Daniel Carando}

\author{Daniel Galicer}

\address{Departamento de Matem\'{a}tica - Pab I,
Facultad de Cs. Exactas y Naturales, Universidad de Buenos Aires,
(1428) Buenos Aires, Argentina, and CONICET.}
\email{dcarando@dm.uba.ar} \email{dgalicer@dm.uba.ar}

\begin{abstract}  In the spirit of the work of Grothendieck, we introduce and study natural symmetric n-fold tensor norms. We prove that there are exactly six natural symmetric tensor norms for $n\ge 3$, a noteworthy difference with the 2-fold case in which there are four. We also describe the polynomial ideals associated to these natural symmetric tensor norms. Using a symmetric version of a result of Carne, we establish which natural symmetric tensor norms preserve Banach algebras.
\end{abstract}


\maketitle
\section*{Introduction}

Alexsander Grothendieck's \emph{``R\'{e}sum\'{e} de la th\'{e}orie m\'{e}trique des produits tensoriels
topologiques''} \cite{Groth53} is considered one of the most influential papers in functional analysis. In this masterpiece, Grothendieck created the basis of what was later known as `local theory', and exhibited the importance of the use of tensor products in the theory of Banach spaces and Operator ideals. As part of his contributions, the \emph{R\'{e}sum\'{e}} contained the list of all \textit{natural} tensor norms. Loosely speaking, this norms come from applying a finite number of basic `operations' to the projective norm.

Grothendieck proved that there were at most fourteen possible natural norms, but he did not know the exact dominations among them, or if there was a possible reduction on the table of natural norms (in fact this was one of the open problems posed in the \emph{R\'{e}sum\'{e}}). Fortunately, this was solved, several years later, thanks to very deep ideas of Gordon and Lewis \cite{GordonLewis1974}.
All this results are now classical and can be found for example in \cite[Section 27]{DefFlo93} and \cite[4.4.2.]{DieFouSwa08}.

Motivated by the increasing interest in theory of symmetric tensor products, we introduce and study natural s-tensor norms of arbitrary order, i.e., tensor norms defined on symmetric tensor products of Banach spaces and which are natural in the sense given by Grothendieck.
Among the fourteen non-equivalent natural 2-fold tensor norms, there are exactly four which are symmetric. The s-tensor version of these four tensor norms are, as expected, the only natural ones for symmetric 2-fold tensor products. One of our main results (Theorem~\ref{son seis}) shows that for $n\ge 3$ we have actually six natural s-tensor norms, a noteworthy difference with the 2-fold case. In this theorem we also describe the maximal ideals of polynomials associated to the natural norms. For this, we use the characterization of the maximal polynomial ideals associated to the injective/projective hull of any s-tensor norm given in Theorem \ref{descripcion}.

The 2-fold tensor norm $w_2'$  is one of the 14 Grothendieck's natural tensor norms, since it is equivalent to $\s \varepsilon_{2} /$ (see \cite[20.17.]{DefFlo93}). In fact, it is also equivalent $ \s / \pi_{2}  \s /$. The same equivalence holds for the analogous 2-fold s-tensor norms. When we pass to $n$-fold tensor products, $ \s / \pi_{n,s}  \s /$ and $\s \varepsilon_{n,s} /$  are no longer equivalent. {In Theorem \ref{no son equivalentes} we prove that, in fact, they are always different on infinite dimensional spaces (the same holds for the norms $/ \pi_{n,s} \s$ and $/ \s \varepsilon_{n,s} / \s$).
In other words, we can say that $w_2'$ splits into two different s-tensor norms when passing to tensor products of order $n\geq3$.
One may wonder which of these s-tensor norms of high order is the most natural extension of the 2-fold symmetric analogue of $w_2'$. We will see that, surprisingly,  two  \emph{good} properties of $w_2'$ are, in a sense, a consequence of it being equivalent to $ \s / \pi_{2,s}  \s /$ rather than to the most simple $\s \varepsilon_{2,s} /$.
The first property we consider is the relationship with its adjoint s-tensor norm. The second is related to the preservation of Banach algebra structures.
Carne in \cite{Carne78} showed that there are exactly four natural 2-fold tensor norms that preserve Banach algebras, two of which are symmetric: $\pi_2$ and $\s \varepsilon_{2} /$. Based on his work we describe in Section~\ref{section preserving} which natural s-tensor norms preserve the algebra structure. We show that the two s-tensor norms preserving Banach algebras are $\pi_{n,s}$ and $ \s / \pi_{n,s}  \s /$. Thus, one  may think that $ \s / \pi_{n,s}  \s /$ is the natural extension of the symmetric analogue of $w_2'$ to tensor norms of higher orders.

All our results on s-tensor norms have their analogous for symmetric tensor norms on full tensor products.

\smallskip
We refer to \cite{DefFlo93} for the theory of 2-fold tensor norms and operator ideals, and to \cite{Flo97,Flo01,Flo02(On-ideals),Flo02(Ultrastability)} for symmetric and full tensor products and polynomial ideals.

\section{Preliminaries}
In this section we recall some definitions and results on the theory of symmetric tensor products and Banach polynomial ideals.

Let $\varepsilon_{n,s}$ and $\pi_{n,s}$ stand for the injective and projective symmetric tensor norms of order $n$ respectively.
We say that  $\beta$ is a s-tensor norm  of order $n$ if $\beta$ assigns to each Banach space $E$ a norm $\beta \big(\; . \;; \otimes^{n,s} E \big)$ on the $n$-fold symmetric tensor product $\otimes^{n,s} E$ such that
\begin{enumerate}
\item $\varepsilon_{n,s} \leq \beta \leq \pi_{n,s}$ on $\otimes^{n,s} E$.
\item $\| \otimes^{n,s} T :  \otimes^{n,s}_\beta E   \to \otimes^{n,s}_\beta F \| \leq \|T\|^n$ for each operator $T~\in~\mathcal{L}(E, F)$.
\end{enumerate}

The s-tensor norm $\beta$ is said to be finitely generated if for all $E \in  BAN$ (the class of all Banach spaces) and $z \in \otimes^{n,s} E$
$$ \beta (z, \otimes^{n,s}E) = \inf \{ \alpha(z, \otimes^{n,s}M) : M \in FIN(E), z \in \otimes^{n,s}M \},$$ where $FIN(E)$ denotes the class of all finite dimensional subspaces of $E$.
In this article we will only work with finitely generated tensor norms and, therefore, all tensor norms will be assumed to be so.

For $\beta$ an s-tensor norm of order $n$, its dual tensor norm $\beta'$ is defined on $FIN$ by

$$   \otimes^{n,s}_{\beta'} M :\overset 1 = \big( \otimes^{n,s}_{\beta} M'  \big)'$$
and extended to $BAN$ as
$$ \beta' ( z, \otimes^{n,s} E ) : = \inf \{ \beta' (z, \otimes^{n,s} M ) : M \in FIN(E),\  z \in \otimes^{n,s} M \}.$$

\smallskip
Similarly, a tensor norm $\alpha$ of order $n$ assigns to every $n$-tuple of Banach spaces $(E_1, \dots, E_n)$ a norm $\alpha \big(\; . \; ; \otimes_{i=1}^n E_i \big)$ on the $n$-fold (full) tensor product $\otimes_{i=1}^n E_i$ such that

\begin{enumerate}
\item $\varepsilon_{n} \leq \alpha \leq \pi_n$ on $\otimes_{i=1}^n E_i$.
\item $\| \otimes_{i=1}^n T_i :  \big( \otimes_{i=1}^n E_i, \alpha \big) \to \big( \otimes_{i=1}^n F_i, \alpha \big) \|  \leq \|T_1\| \dots \|T_n\|$ for each set of operator $T_i~\in~\mathcal{L}(E_i, F_i)$, $i=1, \dots, n$.
\end{enumerate}
Here, $\varepsilon_{n}$ and $\pi_{n}$ stand for the injective and projective full tensor norms of order $n$ respectively.

We often call these tensor norms ``full tensor norms'', in the sense that they are defined on the full tensor product, to distinguish them from the s-tensor norms, that are defined on symmetric tensor products.
The full tensor norm $\alpha$ is finitely generated if for all $E_i \in BAN$ and $z$ in $\otimes_{i=1}^n E_i$
$$ \alpha ( z, \otimes_{i=1}^n E_i ) : = \inf \{ \alpha (z, \otimes_{i=1}^n M_n ) : M_i\in FIN(E_i) (i=1,\dots,n), \ z \in \otimes_{i=1}^n M_i, \}.$$

If $\alpha$ is a full tensor norm of order $n$, then the dual tensor norm $\alpha'$ is defined on FIN  by

$$  \big( \otimes_{i=1}^n M_i, \alpha' \big) :\overset 1 = [\big( \otimes_{i=1}^n M_i', \alpha \big)]'$$
and on BAN  by
$$ \alpha' ( z, \otimes_{i=1}^n E_i ) : = \inf \{ \alpha' (z, \otimes_{i=1}^n M_n ) : M_i\in FIN(E_i) (i=1,\dots,n)\  z \in \otimes_{i=1}^n M_i \}.$$

\smallskip

Let us recall some definitions on the theory of Banach polynomial ideals
\cite{Flo02(On-ideals)}. A \emph{Banach ideal of continuous scalar valued
$n$-homogeneous polynomials} is a pair
$(\mathcal{Q},\|\cdot\|_{\mathcal Q})$ such that:
\begin{enumerate}
\item[(i)] $\mathcal{Q}(E)=\mathcal Q \cap \mathcal{
P}^n(E)$ is a linear subspace of $\mathcal{P}^n(E)$ and $\|\cdot\|_{\mathcal Q(E)}$ (the restriction of $\|\cdot\|_{\mathcal Q}$ to $\mathcal{Q}(E)$) is a norm which makes
$(\mathcal{Q},\|\cdot\|_{\mathcal Q(E)})$ a Banach space.

\item[(ii)] If $T\in \mathcal{L} (E_1,E)$, $p \in \mathcal{Q}(E)$ then $p\circ T\in \mathcal{Q}(E_1)$ and $$ \|
p\circ T\|_{\mathcal{Q}(E_1)}\le  \|P\|_{\mathcal{Q}(E)} \| T\|^n.$$

\item[(iii)] $z\mapsto z^n$ belongs to $\mathcal{Q}(\mathbb K)$
and has norm 1.
\end{enumerate}

Let $(\mathcal{Q},\|\cdot\|_{\mathcal Q})$ be the Banach ideal of continuous scalar valued
$n$-homogeneous polynomials and, for $p \in \mathcal{
P}^n(E)$, define
$$\|p\|_{\Q^{max}(E)}:= \sup \{ \|p|_M\|_{\Q(M)} : M \in FIN(E) \} \in [0, \infty].$$
The maximal hull of $\mathcal{Q}$ is the ideal given by $ \mathcal{Q}^{max} := \{p \in \P : \|p\|_{\Q^{max}} < \infty \}$. An ideal $\mathcal{Q}$ is said to be maximal if $\mathcal{Q} \overset{1}{=} \mathcal{Q}^{max}$.

Also, for $q \in \P$ we define
$$\| q \|_{\Q^{*}}:= \sup \{ | \langle q|_M ,p \rangle | M \in FIN(E), \|p \|_{\Q(M')} \leq 1 \} \in [0,\infty].$$
We will denote $\Q^{*}$ the class of all polynomials
$q$ such that $\| q \|_{\Q^{*}} < \infty.$

The \emph{ s-tensor norm $\gamma$} associated to the Banach polynomial ideal $\Q$ is the unique tensor norm satisfying $$\Q(M) \overset 1 =  \otimes^{n,s}_{\gamma} M,$$ for every finite dimensional space $M$.
The polynomial representation theorem asserts that, if $\Q$ is maximal, then we have $$\Q(E) \overset 1 =  \big( \widetilde{\otimes}^{n,s}_{\gamma'} E \big)',$$ for every Banach space $E$ \cite[3.2]{Flo02(Ultrastability)}. It is not difficult to prove that $\big(\Q^{*}, \| \; \|_{\Q^{*}}\big)$ is  a maximal Banach ideal of continuous $n$-homogeneous polynomials. Moreover, if $\gamma$ is the s-tensor norm associated to the ideal $\Q$ then $ \gamma'$ is the one associated to $\Q^{*}$.

We will sometimes denote by $\Q_\beta$ the maximal Banach ideal of $\beta$-continuous $n$-homogeneous polynomials, that is, $\Q_\beta(E):= \big(\widetilde{\otimes}^{n,s}_\beta E \;\big)'$. We observe that, with this notation, $\Q_\beta$ is the unique maximal polynomial ideal associated to the s-tensor norm $\beta'$.

\smallskip
Let $(\U, \|\; \|_{\U})$ be a Banach ideal of operator. The composition ideal $ \mathcal{Q} \circ \U$ is defined in the following way: a polynomial $p$ belongs to $ \mathcal{Q} \circ \U$ if it admits a factorization
\begin{equation} \label{factorizacion-minimal}
\xymatrix{ E  \ar[rr]^{p} \ar[rd]^{T} & & {\mathbb{K}} \\
& {F} \ar[ru]^{q} & },
\end{equation}
where $F$ is a Banach space, $T \in \U( E,F)$ and $q$ is in $\Q(F)$.
The composition norm is given by $\|p\|_{\Q \circ \U} := \inf \{  \|q\|_{\Q} \|T\|^n_{\U} \}$, where the infimum runs over all possible factorizations as in~(\ref{factorizacion-minimal}).

For $p \in \P$ we define
$$\| p \|_{\Q \circ \U^{-1}}:= \sup \{ \|p \circ T \|_{\Q} : T \in \U, \; \|T \|_{\U}\leq 1,  \; P \circ T  \;\mbox{is defined}\} \in [0,\infty].$$
We will denote $\Q \circ \U^{-1}$ the class of all polynomials
$p$ such that $\| p \|_{\Q \circ \U^{-1}} < \infty.$ It is not
difficult to prove that $\big(\Q \circ \U^{-1}, \| \; \|_{\Q
\circ \U^{-1}}\big)$ is Banach ideal of continuous
$n$-homogeneous polynomials with the property that $p \in \Q
\circ \U^{-1}$ if and only if $p \circ T \in \Q$ for all $T \in
\U$. In other words, $\Q \circ \U^{-1}$ is the largest ideal
satisfying $(\Q \circ \U^{-1}) \circ \U \subset \Q$.

By $\P_f$ we will denote the class of finite type polynomials.
{We say that a polynomial ideal $\Q$ is \textit{accessible} if the following condition holds:
for every Banach space $E$, $q \in \p_f(E)$ and $\varepsilon > 0$, there is a closed finite codimensional space $L \subset E$ and $p\in \p(E/L)$ such that $q = p \circ Q_{L}^E$ (where $Q_{L}^E$ is the canonical quotient map) and $\|p\|_{\q}~\leq~(1 + \varepsilon)~\|q\|_{\q}.$}

Let $M$ be a finite dimensional Banach space. For $p \in P(M)$ and  $q\in \P(M')$, we denote by $\langle q,p \rangle$ the trace-duality of polynomials, defined for $p=(x')^n$ and $q=x^n$ as
$$ \langle  p, q \rangle = x'(x)^n,$$
and extended by linearity \cite[1.13]{Flo97}.

{Finally, a surjective mapping $T: E \to F$ is called \emph{a metric surjection} if $$\|Q(x)\|_F=\inf\{\|y\|_E : Q(y)=x \},$$ for all $x \in E$. As usual, a mapping $I : E \to F$ is called \emph{isometry} if $\|Ix\|_F = \|x\|_E$ for all $x \in E$. We will use the notation  $\overset 1 \twoheadrightarrow$ and  $\overset 1 \hookrightarrow$ to indicate a metric surjection or an isometry, respectively.
We also write $E\overset 1=F$ if $E$ and $F$ are isometrically isomorphic Banach spaces (i.e. there exist an surjective isometry $I : E \to F$).}
For a Banach space $E$ with unit ball $B_E$, we call the mapping $Q_E : \ell_1(B_E) \overset 1 \twoheadrightarrow E$ given by
\begin{equation}\label{canonical quotient}Q_E\big((a_x)_{x\in B_E} \big)= \sum_{x\in B_E} a_x x\end{equation}
the \emph{canonical quotient mapping}.
Also, we consider the \emph{canonical embedding} $I_E: E\to \ell_\infty(B_{E'})$ given by \begin{equation}\label{canonical embedding}I_E(x)=\big(x'(x)\big)_{x'\in B_{E'}}.\end{equation}

\section{Projective and Injective hulls of an s-tensor norm}\label{adjoint of the composition}

In this section we will define the projective and injective hulls of an s-tensor norm and describe their associated maximal Banach ideals of polynomials.

The projective and injective associates (or hulls) of $\beta$ will be  denoted, by extrapolation of the 2-fold full case, as $\s \beta /$ and $/ \beta \s$ respectively. The projective associate of $\beta$ will be the (unique) smallest projective tensor norm greater than $\beta$. Following some ideas from \cite[Theorem 20.6.]{DefFlo93} we have
$$ \otimes^{n,s} Q_E \colon \otimes^{n,s}_\beta \ell_1(E) \overset 1 \twoheadrightarrow  \otimes^{n,s}_{\s \beta /}  E,$$ where $Q_E : \ell_1(B_E) \twoheadrightarrow E$ is the canonical quotient map defined in (\ref{canonical quotient})

The injective associate of $\beta$ will be the (unique) greatest injective tensor norm smaller than $\beta$.  As in \cite[Theorem 20.7.]{DefFlo93} we get,
$$ \otimes^{n,s} I_E \colon  \otimes^{n,s}_{/ \beta \s} E  \overset 1 \hookrightarrow  \otimes^{n,s}_\beta \ell_{\infty}(B_{E'}),$$  where $I_E$ is the canonical embedding (\ref{canonical embedding}).

The projective and injective associates for a full tensor norm $\alpha$ can be defined in a similar way and satisfy
$$ \big( \otimes_{i=1}^n \ell_1(E_i),  \alpha  \big) \overset 1 \twoheadrightarrow \big( \otimes_{i=1}^n  E_i,   \s \alpha /  \big), \;\;\; \big( \otimes_{i=1}^n E_i , / \alpha \s \big) \overset 1 \hookrightarrow  \big( \otimes_{i=1}^n \ell_{\infty}(B_{E_i'}),  \alpha  \big).$$

The following duality relations for an s-tensor norm $\beta$ or a full tensor norms $\alpha$ are obtained proceeding as in
\cite[Proposition 20.10.]{DefFlo93}:
$$  (/ \beta \s)' = \s \beta' /, \; \; \;   (\s \beta /)' = / \beta' \s,\; \; \; (/ \alpha \s)' = \s \alpha ' /, \; \; \; (\s \alpha /)' = /\alpha '\s.  $$

Just as in \cite[Corollary 20.8]{DefFlo93}, if $E$ is an $\mathcal{L}_{1,\lambda}$ space for every $\lambda >1$, then $\beta$ and $ \s \beta /$ coincide (isometrically) on $\otimes^{n,s} E$. On the other hand, if $E$ is an $\mathcal{L}_{\infty,\lambda}$ space for every $\lambda >1$, then $\beta$ and $ / \beta \s$ coincide in $\otimes^{n,s} E$.
A similar result holds for full tensor norms: if $E_1, \dots, E_n$ are $\mathcal{L}_{1,\lambda}$ spaces for every $\lambda >1$ then $\alpha$ and $ \s \alpha /$ are equal on $\otimes_{i=1}^n E_i$. On the other hand, if $E_1, \dots, E_n$ are $\mathcal{L}_{\infty,\lambda}$ spaces for every $\lambda >1$ then $\alpha$ and $ / \alpha \s$ coincide in $\otimes_{i=1}^n E_i$.

It is not difficult to prove that an $n$-homogeneous polynomial $p$ belongs to $\q_{\s \beta /}(E)$ if and only if $p \circ Q_E \in \q_{\beta}(\ell_1(B_E))$.
Moreover,
\begin{equation}\label{identproyec}
\|p\|_{\q_{\s \beta /}(E)}= \|p \circ Q_E \|_{\q_{\beta}(\ell_1(B_E))}.
\end{equation}
On the other hand, an $n$-homogeneous polynomial $p$ belongs to $\q_{/ \beta \s}(E)$ if and only if there
exist an $n$-homogeneous polynomial $\overline{p} \in \q_{\beta}(\ell_{\infty}(B_{E'}))$ such that $\overline{p} \circ I_E = p$ and
\begin{equation}\label{identinyec}
\|p\|_{\q_{/ \beta \s}(E)}=
\|\overline{p} \|_{\q_{ \beta }(\ell_{\infty}(B_{E'}))}.
\end{equation}
 In other words, $/ \beta \s$-continous polynomials are those that can be extended to $\beta$-continuous polynomials on $\ell_{\infty}(B_{E'})$.
As a consequence, the injective associate of the projective s-tensor norm, $/ \pi_{n,s} \s$, is the predual norm of the ideal of extendible polynomials $\P_e$ (see \cite{Carando99}, and also \cite{KirRy98}, where this norm is constructed in a different way). The norm $/ \pi_{n,s} \s$ usually appears in the literature denoted by $\eta$.

The description of the $n$-linear forms belonging to $\big( \otimes_{i=1}^n  E_i,   \s \alpha /  \big)'$ or to $\big( \otimes_{i=1}^n  E_i,  /  \alpha \s  \big)'$ is analogous to that for polynomials.

The following result describes the maximal Banach ideal of polynomials associated to the projective/injective hull of an s-tensor norm in terms of  composition ideals.
\begin{theorem}\label{descripcion}
Let $\beta$ be an s-tensor norm of order $n$. We have the following identities:
$$ \Q_{/\beta\s} \overset{1}{=} \Q_\beta \circ \mathcal{L}_\infty \;\; \mbox{and} \;\; \Q_{\s \beta /} \overset{1}{=}  \Q_\beta \circ (\mathcal{L}_1)^{-1}.$$
\end{theorem}

To prove this, we will need a polynomial version of the Cyclic Composition Theorem \cite[Theorem 25.4.]{DefFlo93}.

\begin{lemma}\label{cyclic Composition Theorem}
Let $(\Q_1, \| \; \|_{\Q_1})$, $(\Q_2, \| \; \|_{\Q_2})$ be two
Banach ideals of continuous $n$-homogeneous polynomials and $(\U,
\|\; \|_{\U})$ a Banach operator ideal with $(\U^{dual},\|\; \|_{\U^{dual}})$ right-accessible. If $$ \Q_1 \circ
\U \subset \Q_2,$$ and $\|\; \|_{\Q_2} \leq k \|\; \|_{\Q_1 \;
\circ \; \U}$ for some positive constant $k$, then we have $$ \Q_2^{*} \circ \U^{dual} \subset \Q_1^{*},$$
with $\| \; \|_{\Q_1^{*}} \leq k \| \; \|_{\Q_2^{*} \; \circ \;
\U^{dual}}$.
\end{lemma}

\begin{proof}
Fix $q \in\Q_2^{*} \circ \U^{dual} (E)$, $M \in FIN(E)$ and $p \in \Q_1(M')$ with $\| p \|_{\Q_1(M')} \leq 1$. For $\varepsilon > 0$, we take $T \in \U^{dual}(E,F)$ and $q_1 \in \Q_2^{*}(F)$ such that $q= q_1 \circ T$ and $$\|q_1\|_{\Q_2^{*}} \|T\|^n_{\U^{dual}} \leq (1 + \varepsilon ) \|q\|_{\Q_2^{*} \; \circ \; \U^{dual}}.$$
Since $(\U^{dual},\|\; \|_{\U^{dual}})$ is right-accessible, by definition \cite[21.2]{DefFlo93} there are $N \in FIN(F)$ and $S \in \U^{dual}(M,N)$ with $\|S\|_{\U^{dual}} \leq (1+\varepsilon)\|T{|_M}\|_{\U^{dual}} \leq (1+\varepsilon)\|T\|_{\U^{dual}}$ satisfying
\begin{equation}\label{diamaccess}
\xymatrix{ M  \ar[rr]^{T{|_M}} \ar[rrd]^{S} & & {F} \\
& & {N} \ar@{^{(}->}[u]^{i_N}},
\end{equation}
Thus, since the adjoint  $S^{*}$ of $S$ belongs to $\U(N',M')$, we have
\begin{align*}
\big| \langle q|{_M}, p \rangle \big| &  = \big| \langle q_1 \circ T{|_M}, p \rangle \big|  = \big| \langle {q_1} \circ i_N \circ S , p \rangle \big| \\
& = \big| \langle {q_1} \circ i_N , p \circ S^{*} \rangle \big|  \leq  \| {q_1} \circ i_N\|_{\Q_2^{*}} \; \|p  \circ S^{*} \|_{\Q_2}  \\
& \leq  k \| q_1 \|_{\Q_2^{*}} \; \|p  \circ S^{*} \|_{\Q_1 \; \circ \; \U}  \leq  k \| q_1 \|_{\Q_2^{*}} \; \|p \|_{\Q_1} \; \|S^{*}\|^n_{\U}  \\
& \leq  k \| q_1 \|_{\Q_2^{*}} \; \|S\|^n_{\U^{dual}}  \leq  k (1 + \varepsilon)^{n}  \| q_1 \|_{\Q_2^{*}} \; \|T\|^n_{\U^{dual}} \\
& \leq  k (1 + \varepsilon)^{n+1}  \|q\|_{\Q_2^{*} \; \circ \; \U^{dual}}.
\end{align*}
This holds for every $M \in FIN(E)$ and every $p \in \Q_1(M')$ with $\| p \|_{\Q_1(M')} \leq 1$, thus $q \in \Q_1^{*}$ and $\| q \|_{\Q_1^{*}} \leq k (1+\varepsilon) \| q \|_{\Q_2^{*} \; \circ \; \U^{dual}}$. Since $\varepsilon > 0$ is arbitrary we get $\| q \|_{\Q_1^{*}} \leq k \| q \|_{\Q_2^{*} \; \circ \; \U^{dual}}$.
\end{proof}

Notice that the condition of $(\U^{dual},\|\; \|_{\U^{dual}})$ being right-accessible is fulfilled whenever $(\U,\|\; \|_{\U})$ is a maximal left-accessible Banach ideal of operators \cite[Corollary 21.3.]{DefFlo93}.

\begin{proposition}\label{adjunto de la composicion}
Let $(\Q, \| \; \|_{\Q})$ a Banach ideal of continuous
$n$-homogeneous polynomials and $(\U, \|\; \|_{\U})$ a Banach
ideal of operators. If $\U$ is maximal and accesible (or $\U$ and $\U^{dual}$ are both right-accesible), then $$(\Q \circ \U )^{*} \overset{1}{=}
\Q^{*} \circ ({\U^{dual}})^{-1}.$$
\end{proposition}

\begin{proof}
Lemma~\ref{cyclic
Composition Theorem} applied to the inclusion $\Q \circ \U \subset \Q \circ \U $ implies that $ (\Q \circ \U)^{*} \circ
\U^{dual} \subset \Q^{*}$. Therefore, $(\Q \circ \U )^{*} \subset
\Q^{*} \circ ({\U^{dual}})^{-1}$ and $\| \; \|_{\Q^{*} \circ
({\U^{dual}})^{-1}} \leq \| \; \|_{(\Q \circ \U )^{*}}$.

For the reverse inclusion we proceed similarly as in  proof of Lemma~\ref{cyclic Composition
Theorem}.
Fix $q \in \Q^{*} \circ ({\U^{dual}})^{-1}(E)$, $M \in FIN(E)$ and
$p \in \Q \circ \U(M')$ with  $\|p\|_{ \Q \circ \U(M')}\leq
1$. For $\varepsilon > 0$, we take $T \in \U(M',F)$ and $p_1 \in \Q(F)$  such
that $p=p_1 \circ T$ and $\|p_1\|_{\Q} \|T\|^n_{\U} \leq (1 + \varepsilon )$.
Since $(\U, \|\; \|_{\U})$ is accessible, there are $N \in FIN(F)$ and $S \in \U(M',N)$ with $$\|S\|_{\U^{dual}} \leq (1+\varepsilon)\|T{|_M}\|_{\U^{dual}} \leq (1+\varepsilon)\|T\|_{\U}$$ satisfying $T{|_M} = i_N \circ S$.
Note that
$S^{*} \in \U^{dual}$ and
$\|S^{*}\|_{\U^{dual}} \leq (1+\varepsilon) \|T\|_{\U}$. Thus, $q_{|_M}
\circ (S)^{*} \in \Q^{*}$ and $\|q{|_M} \circ
(S)^{*}\|_{\Q^{*}} \leq (1+\varepsilon)^n \|q\|_{\Q^{*} \circ
({\U^{dual}})^{-1}} \|T\|_{\U}^n$. Now we have:
\begin{align*}
\big| \langle q{|_M}, p \rangle \big| &  = \big| \langle q{|_M}, p_1 \circ T \rangle \big|  = \big| \langle q{|_M} , p_1 \circ i_N \circ S  \rangle \big| \\
& \leq \big| \langle q{|_M} \circ S^{*} , p_1 \circ i_N   \rangle \big|  \leq \|q{|_M} \circ S^{*} \|_{\Q^{*}} \;
\|p_1 \circ i_N\|_{\Q} \\
& \leq (1+\varepsilon)^n \|q\|_{\Q^{*} \circ ({\U^{dual}})^{-1}} \; \|p_1\|_{\Q}  \; \|T\|_{\U}^n
\\
& \leq (1 + \varepsilon)^{n+1} \|q\|_{\Q^{*} \circ ({\U^{dual}})^{-1}}.
\end{align*}
This holds for every $M \in FIN(E)$, every $p \in \Q \circ \U(M')$
with $\|p\|_{ \Q \circ \U(M')}\leq 1$ and every $\varepsilon
> 0$. As a consequence,  $q \in (\Q \circ \U )^{*}$ and $ \| q \|_{(\Q \circ \U )^{*}} \leq \| q \|_{\Q^{*} \circ
({\U^{dual}})^{-1}} $.
\end{proof}

Now we can prove Theorem \ref{descripcion}:
\begin{proof}(Theorem \ref{descripcion})
We have already mentioned that any  $p \in \Q_{/\beta\s}(E)$ extends to a polynomial $\overline p$ defined on
$\ell_\infty(B_{E'})$ with $\|\overline{p}\|_{\Q_{\beta}(\ell_\infty(B_{E'}))}=\|p\|_{\Q_{/\beta\s}(E)}$.
Therefore, $p$ belongs to $\Q_\beta \circ \mathcal{L}_\infty$ and
$$\|p\|_{\Q_\beta \circ \mathcal{L}_\infty} \leq
\|\overline{p}\|_{\Q_{\beta}(\ell_\infty(B_{E'}))} \|i\|^{n} =
\|p\|_{\Q_{/\beta\s}(E)}.$$ On the other hand, for $p \in \Q_\beta
\circ \mathcal{L}_\infty$ and $\varepsilon > 0$ we can take $ T \in
\mathcal{L}_\infty(E,F)$ and $q \in \Q_\beta(F)$ such that $p=q
\circ T$ and $\|q\|_{\Q} \|T\|_{\mathcal{L}_\infty}^n \leq (1+
\varepsilon) \|p\|_{\Q_\beta \circ \mathcal{L}_\infty}$. We
choose $R \in \mathcal{L}(E,L_\infty(\mu))$ and $S \in
\mathcal{L}(L_\infty(\mu)),F'')$ factoring $J_F \circ T: E \to
F''$ with $\|R\| \|S\| \leq (1+\varepsilon)
\|T\|_{\mathcal{L_\infty}}$. Also, since $\Q_\beta$ is a maximal polynomial ideal, its canonical extension $\overline{q}: F'' \to
\mathbb{K}$ belongs to  $\Q_\beta$ and satisfy
$\|\overline{q}\|_{\Q_\beta}=\|q\|_{\Q_\beta}$ \cite{CarGal-extending}. We have the following commutative
diagram:
\begin{equation*}
\xymatrix{&  E \ar[rr]^{p} \ar[d]^{T} \ar[dl]^{R} & &  {\mathbb{K}}\\
L_\infty(\mu) \ar[dr]^{S}  & F \ar[urr]^{q} \ar@{^{(}->}[d]^{J_F} & & \\
& {F''} \ar@/_1.2pc/[uurr]^{\overline{q}} & &}.
\end{equation*}
Since $\overline{q} \circ S \in \Q_{\beta} (L_\infty (\mu) )
\overset{1}{=} \Q_{/\beta\s} (L_\infty (\mu))$ we have
\begin{align*}
\|p\|_{\Q_{/ \beta \s}} & \leq \| \overline{q} \circ
S \|_{\Q_{/ \beta \s}} \; \|R\|^n \\
& = \| \overline{q} \circ S\|_{\Q_{\beta}} \; \|R\|^n \\
& \leq \|\overline{q}\|_{\Q_{\beta}} \; \|S\|^n \; \|R\|^n \\
& \leq (1+\varepsilon)^n \; \|q \|_{\Q_\beta} \;
\|T\|^n_{\mathcal{L}_\infty}  \\
& \leq (1+\varepsilon)^{n+1} \|p\|_{\Q_{ \beta } \circ
\mathcal{L}_\infty}.
\end{align*}
Thus, $ \Q_{/\beta\s} \overset{1}{=} \Q_\beta \circ
\mathcal{L}_\infty.$

Now we show the second identity. First notice that $\mathcal{L}_1 =\mathcal{L}_\infty^{dual}$ (this follows, for example, from Corollary 3 in \cite[ 17.8.]{DefFlo93} and the information on the table in \cite[27.2.]{DefFlo93}). Since $\mathcal{L}_\infty$ is maximal and accessible \cite[Theorem 21.5.]{DefFlo93}, an application of
Proposition~\ref{adjunto de la composicion} to the equality
$\Q_{/\beta'\s} \overset{1}{=} \Q_{\beta'} \circ
\mathcal{L}_\infty$ gives
 $\Q_{\s \alpha /} = \Q_\alpha \circ
\mathcal{L}_1^{-1}$ with $\| \; \|_{\Q_\alpha \circ
\mathcal{L}_1^{-1}} = \| \; \|_{\Q_{\s \alpha /}}$.
\end{proof}

\section{Symmetric natural tensor norms of order $n$}\label{seccion naturales}
In \cite{Groth53} Grothendieck defined natural 2-fold norms as those that can be obtained from $\pi_2$ by a finite number of the following operations: right injective hull, left injective hull, right projective hull, left projective hull and adjoint.
The aim of this section is to define and study  natural symmetric tensor norms of arbitrary order, in the spirit of Grothendieck's norms.

\begin{definition}
Let $\beta$ be an s-tensor norm of order $n$. We say that $\beta$ is a natural s-tensor norm if $\beta$ is obtained from $\pi_{n,s}$ with a finite number of the operations $\setminus  \ /$, $/ \ \setminus$, $'$.
\end{definition}

For (full) tensor norms of order 2, there are exactly four natural norms that are symmetric \cite[Section 27]{DefFlo93}. It is easy to show that the same holds for s-tensor norms of order 2 (see the  proof of Theorem~\ref{son seis}).
These are $\pi_{2,s}$, $\varepsilon_{2,s}$, $/\pi_{2,s}\s$ and $\s \varepsilon_{2,s} /$, with the same dominations as in the full case. It is important to mention that, for $n=2$, $\s \varepsilon_{n,s} /$  and $\s /\pi_{n,s}\s /$, or equivalently, $/\pi_{n,s}\s$  and  $ / \s \varepsilon_{n,s} / \s$, coincide.
However, for $n\ge 3$, we have the following.
\begin{theorem}\label{son seis}
For $n\ge 3$, there are exactly 6 different natural symmetric s-tensor norms. They can be arranged in the following way:

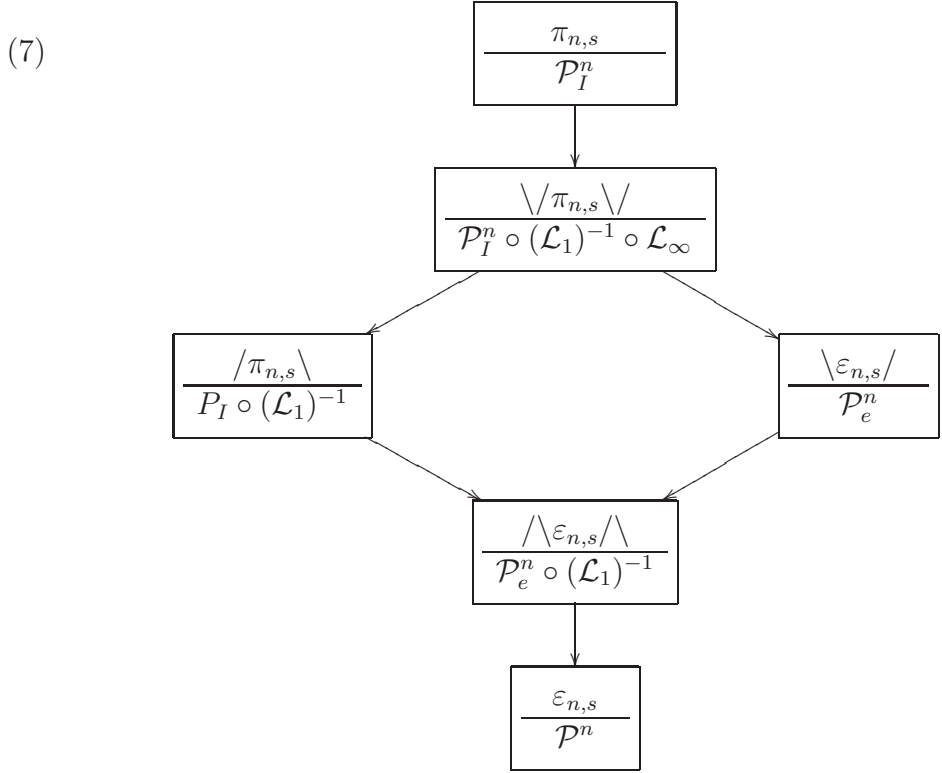
\begin{equation}\label{normasnaturales}
\xymatrix{ &*+[F]{ \begin{array}{c}
                  \;\;\;\;\;\;   \pi_{n,s} \;\;\;\;\;\; \\
\hline
\P_I
                  \end{array}}\ar[d]
   & \\
&*+[F]{
\begin{array}{c}  \;  \; \s / \pi_{n,s} \s / \; \;\\ \hline \P_I  \circ (\mathcal{L}_1)^{-1} \circ \mathcal{L}_\infty \end{array}}
\ar[ld]\ar[rd] & \\
*+[F]{\begin{array}{c} \; \; / \pi_{n,s} \s \; \; \\ \hline P_I \circ (\mathcal{L}_1)^{-1} \end{array}} \ar[rd] &  & *+[F]{\begin{array}{c}  \; \; \s \varepsilon_{n,s} / \; \; \\ \hline \P_e \end{array}} \ar[ld] \\
&*+[F]{\begin{array}{c}  \; \; / \s \varepsilon_{n,s} / \s \;  \; \\ \hline  \P_e \circ (\mathcal{L}_1)^{-1} \end{array}}\ar[d] & \\
&*+[F]{\begin{array}{c}  \; \; \varepsilon_{n,s} \;  \; \\ \hline \P \end{array}}& \\
}
\end{equation}
where $\beta \to \gamma$ means that $\beta$ dominates $\gamma$. There are no other dominations than those showed in the scheme. Below each tensor norm we find its associated maximal polynomial ideal. %
\end{theorem}

Before we prove the Theorem, we need some previous results and definitions.

\begin{lemma}\label{lemma1}
Let  $\beta$ be an s-tensor norm of order $n$. Then $ \s / \s / \beta \s / \s / = \s / \beta \s /$
and $ / \s / \s \beta / \s / \s = / \s \beta / \s.$
\end{lemma}

\begin{proof} It is enough to show the ``$\le$'' inequalities in both equations, since the reverse ones follow by duality.
Since $ / \s / \beta \s / \s \leq \s / \beta \s /$ and  $\s / \beta \s /$ is projective, we can conclude that
$ \setminus / \setminus / \beta \setminus / \setminus / \leq  \setminus / \beta \setminus /.$
For the second inequality, we have $ / \s \beta / \s \leq \s \beta / $ and, by the projectiveness of $\s \beta /$,  we obtain
$\s / \s \beta / \s /\leq \s \alpha /.$ So the corresponding injective hulls satisfy the same inequality, as desired.
\end{proof}

Let $\alpha$ be a full tensor norm of order $n$. We will denote by $\underline{\alpha}$ the full tensor norm of order $n-1$ given by $$\underline{\alpha}(z, \otimes_{i=1}^{n-1} E_i) := \alpha(z \otimes 1,  E_1 \otimes \dots \otimes  E_{n-1} \otimes \mathbb{C}),$$ where  $z \otimes 1 := \sum_{i=1}^m x_1^i \otimes \dots x_n^i \otimes 1$, for $z = \sum_{i=1}^m x_1^i \otimes \dots x_n^i$ (this definition can be seen as dual to some ideas on  \cite{BotBraJunPel06} and \cite{CarDimMurN}).

\begin{lemma}\label{bajarelgrado}
For any tensor norm $\alpha$, we have: $ \underline{(/ \alpha \s)} = / \underline{\alpha} \s $ and $ \underline{(\s \alpha /)} = \s \underline{\alpha} / $. Also, if $\alpha$ and $\gamma$ are full tensor norms and there exists $C>0$ such that $\alpha \le C \gamma$, then $\underline{\alpha}\le C\underline{\gamma}$.
\end{lemma}

\begin{proof}
Let $z \in  \otimes_{i=1}^n E_i$.
For the first statement, if $I_i : E_i \to \ell_{\infty}(B_{{E'_i}})$ are the canonical embeddings, we have
$$
\begin{aligned}
/ \underline{\alpha} \s \big ( z, E_1 \otimes \dots \otimes E_{n-1} \big) & =  \underline{\alpha} \big( \otimes_{i=1}^n I_i (z), \ell_\infty(B_{E_1'}) \otimes \dots \otimes \ell_\infty(B_{E_{n-1}'}) \big) \\
& = \alpha \big( \otimes_{i=1}^n I_i (z) \otimes 1, \ell_\infty(B_{E_1'}) \otimes \dots \otimes \ell_\infty(B_{E_{n-1}'}) \otimes \mathbb{C} \big) \\
& = / \alpha \s \big( z \otimes 1, E_1 \otimes \dots \otimes E_{n-1} \otimes \mathbb{C} \big) \\
& = \underline{(/ \alpha \s)} \big( z, E_1 \otimes \dots \otimes E_{n-1} \big).
\end{aligned}
$$
For the second statement, if $Q_i : \ell_1(B(E_i)) \twoheadrightarrow E_i$ are the canonical quotient mappings, we obtain
$$
\begin{aligned}
\s \underline{\alpha} / \big( z, E_1 \otimes \dots E_{n-1} \big) & = \inf_{\{t \; / \; \otimes_{i=1}^{n-1} P_i (t)= z \}}  \underline{\alpha} \big( t, \ell_1(B_{E_1}) \otimes \dots \otimes \ell_1(B_{E_{n-1}}) \big) \\
& = \inf_{\{t \; / \; \otimes_{i=1}^{n-1} P_i (t)= z \}}   \alpha \big( t \otimes 1, \ell_1(B_{E_1}) \otimes \dots \otimes \ell_1(B_{E_n}) \otimes \mathbb{C} \big) \\
& = \inf_{ \{t \; / \; (P_1 \otimes \dots P_{n-1} \otimes id_{\mathbb{C}}) (t \otimes 1) \; = \; z \otimes 1 \}}  \alpha \big( t \otimes 1, \ell_1(B_{E_1}) \otimes \dots \otimes \ell_1(B_{E_{n-1}}) \otimes \mathbb{C} \big) \\
& = \s \alpha / \big(z \otimes 1, E_1 \otimes \dots \otimes E_{n-1} \otimes \mathbb{C} \big) \\
& = \underline{(\s \alpha /)} \big(z , E_1 \otimes \dots \otimes E_{n-1}  \big).
\end{aligned}
$$
The third statement is immediate.
\end{proof}

Floret in \cite{Flo01(extension)} showed that for every s-tensor norm $\beta$ of order $n$ there exist a full tensor norm $\Phi(\beta)$ of order $n$ which is equivalent to $\beta$ when restricted on symmetric tensor products (i.e. there is a constant $d_n$ depending only on $n$ such that $d_n^{-1} \Phi(\beta)|_s  \leq \beta \leq d_n \Phi(\beta)|_s$ in $\otimes^{n,s}E$ for every Banach space $E$). As a consequence, a large part of the isomorphic theory of norms on symmetric tensor products can be deduced from the theory of ``full'' tensor norms, which is usually easier to handle and has been more studied.

\begin{lemma} \label{Propo phi}
Let  $\beta$ be an s-tensor norm of order $n$. Then $\Phi(/ \beta \setminus) $ and $ /\Phi(\beta)\setminus$ are equivalent s-tensor norms. Also,  $\Phi(\setminus \beta /)$ and $ \setminus \Phi(\beta)/$ are equivalent s-tensor norms.
\end{lemma}
\begin{proof}For simplicity, we consider the case $n=2$, the proof of the general case being completely analogous. The definition of the injective associate gives
$$ E_1 \otimes_{/\Phi(\beta)\setminus} E_2 \overset 1 \hookrightarrow \ell_\infty(B_{E_1'}) \otimes_{\Phi(\beta)} \ell_\infty(B_{E_2'}).$$
Take $x_1, \dots, x_r\in E_1$ and  $y_1, \dots, y_r\in E_2$ and
let $I_i: E_i \to \ell_\infty(B_{E_i'})$ be the canonical embeddings (\ref{canonical embedding}).
Following the notation in~\cite{Flo01(extension)}, we have:
$$
\begin{aligned}
 / \Phi(\beta) \setminus & \big( \sum_{j=1}^r x_j \otimes y_j \big) = \Phi(\beta) \big(\sum_{j=1}^r I_1({x_j}) \otimes I_2(y_j), \ell_\infty(B_{E_1'}) \otimes \ell_\infty(B_{E_2'}) \big) \\
& = \sqrt{2} K_2(\beta)^{-1} \beta \big( \sum_{j=1}^r (I_1(x_j),0) \vee (0, I_2(y_j)), \otimes^{2,s} \{ \ell_\infty(B_{E_1'}) \oplus_2 \ell_\infty(B_{E_2'}) \}\big) \\
&\asymp \sqrt{2} K_2(\beta)^{-1} \beta \big( \sum_{j=1}^r (I_1(x_j),0) \vee (0, I_2(y_j)), \otimes^{2,s} \{ \ell_\infty(B_{E_1'}) \oplus_\infty \ell_\infty(B_{E_2'})\} \big) \\
&= \sqrt{2} K_2(\beta)^{-1} / \beta \setminus \big( \sum_{j=1}^r (I_1(x_j),0) \vee (0, I_2(y_j)), \otimes^{2,s} \{ \ell_\infty(B_{E_1'}) \oplus_\infty \ell_\infty(B_{E_2'}) \} \big) \\
&\asymp \sqrt{2} K_2(\beta)^{-1} / \beta \setminus \big( \sum_{j=1}^r (I_1(x_j),0) \vee (0, I_2(y_j), \otimes^{2,s} \{ \ell_\infty(B_{E_1'}) \oplus_2 \ell_\infty(B_{E_2'}) \} \big) \\
& =\sqrt{2} K_2(\beta)^{-1} / \beta \setminus \big( \sum_{j=1}^r (x_j,0) \vee (0, y_j), \otimes^{2,s} \{ E_1 \oplus_2 E_2 \}) \\
& = \Phi(/ \beta \setminus) ( \sum_{j=1}^r x_j \otimes y_j),
\end{aligned}
$$
where $\asymp$ means that the two expressions are equivalent with universal constants.
The second equivalence follows from the first one by duality, since  by~\cite[Theorem 2.3.(8)]{Flo01(extension)} we have $ \Phi(\setminus \beta /) =  \Phi((/ \beta' \setminus)') \sim \Phi(/ \beta' \setminus)' \sim /\Phi( \beta' )\setminus' =  \setminus \Phi(\beta')' / \sim \setminus \Phi(\beta) / $.
\end{proof}

\begin{lemma}\label{inyectiva-proyectiva}
No injective norm $\beta$ can be equivalent to a projective norm $\delta$.
\end{lemma}

\begin{proof}
If they were equivalent, we would have $\s \varepsilon_{n,s} / \leq  \s \beta / \leq C_1 \delta \leq C_2 \beta \leq C_2 / \pi_{n,s} \s $.
By the fact that $\Phi$ respects inequalities \cite[Theorem 2.3.(4)]{Flo01(extension)}, the equivalences ${\pi_n}|_s \sim \pi_{n,s}$ and ${\varepsilon_n}|_s \sim \varepsilon_{n,s}$ and \cite[Theorem 2.3.(9)]{Flo01(extension)}, we obtain $\s \varepsilon_n / \leq D  / \pi_n \s $, for some constant $D$. By the obvious identities $\underline{\varepsilon_{n+1}}= \varepsilon_n$, $\underline{\pi_{n+1}} = \pi_n$ and applying Lemma~\ref{bajarelgrado} $n-2$ times we get $ \s \varepsilon_2 / \sim w_2' \leq D  / \pi_2 \s \sim w_2 $, a contradiction.
\end{proof}

\medskip

Now we are ready to prove  Theorem~\ref{son seis}.
\begin{proof} (of Theorem~\ref{son seis})
As a first step, we show  that $\pi_{2,s}$, $\varepsilon_{2,s}$, $/\pi_{2,s}\s$ and $\s \varepsilon_{2,s} / $ are the non-equivalent natural s-tensor norms for $n=2$. We can see in \cite[Chapter 27]{DefFlo93} that $\pi_{2}$, $\varepsilon_{2}$, $/\pi_{2}\s$ and $\s \varepsilon_{2} / $ are the only natural 2-fold tensor norms that are symmetric. So we can use Lemma~\ref{Propo phi}, the fact that $\Phi(\pi_{2,s})$ is equivalent to $ \sim \pi_2$ on the symmetric tensor product to conclude our claim. This also shows the following dominations: $\varepsilon_{2,s} \leq \s \varepsilon_{2,s} / \leq /\pi_{2,s}\s \leq \pi_{2,s}$.

To prove that, for $n\ge 3$,  all the possible natural $n$-fold s-tensor norms are listed in (\ref{normasnaturales}), it is enough to show that $ / \s / \pi_{n,s} \s / \s$ coincides with $/ \pi_{n,s} \s $. But this follows from the first equality in Lemma~\ref{lemma1} and the projectiveness of  $\pi_{n,s}$, which means that $\pi_{n,s} = \s \pi_{n,s} /$.

\medskip
Now we see that the listed norms are all different.
First,  $/ \pi_{n,s} \s$ and $ \s / \pi_{n,s} \s /$ cannot be equivalent by Lemma~\ref{inyectiva-proyectiva}. Analogously,  $\s \varepsilon_{n,s} / $ is not equivalent to $ / \s \varepsilon_{n,s} / \s $. Until now, everything works just as in the case $n=2$. The difference appears when we consider the relationship between $ \s / \pi_{n,s} \s /$ and $ \s \varepsilon_{n,s} /$.


For $n\ge 3$,  it is shown in \cite{CarDim07,Per04(4-linear),Varapoulus(1975)} that $ / \pi_{n,s} \s $ and $\varepsilon_{n,s}$ cannot be equivalent in any infinite dimensional Banach space. Since on  $\otimes^{n,s} \ell_1$ the s-tensor norm  $ \s / \pi_{n,s} \s /$ coincides with $ / \pi_{n,s} \s $ and $ \s \varepsilon_{n,s} /$ coincides with $\varepsilon_{n,s}$, it follows that   $ \s / \pi_{n,s} \s /$ and $ \s \varepsilon_{n,s} /$ are not equivalent s-tensor norms, since they are not equivalent on $\otimes^{n,s} \ell_1$ (we will actually see in Theorem~\ref{no son equivalentes} that $ \s / \pi_{n,s} \s /$ and $ \s \varepsilon_{n,s} /$ cannot be equivalent on any infinite dimensional Banach space).

By duality, conclude that the six listed norms in Theorem~\ref{son seis} are different.

\medskip

It is clear that all the dominations presented in (\ref{normasnaturales}) hold, so we must show that $/ \pi_{n,s} \s$ does not dominate $ \s \varepsilon_{n,s} / $ nor $ \s \varepsilon_{n,s} / $ dominates $/ \pi_{n,s} \s$.  Note that the inequality $ / \pi_{n,s} \s \leq C \s \varepsilon_{n,s} /$ would imply the equivalence between  $ / \pi_{n,s} \s $ and $   \varepsilon_{n,s} $ on $\otimes^{n,s} \ell_1$, which is impossible by the already mentioned result of \cite{CarDim07,Per04(4-linear),Varapoulus(1975)}. Finally, reasoning as in the proof of Lemma~\ref{inyectiva-proyectiva}, we also have $ \s \varepsilon_{n,s} / $ does not dominate $/ \pi_{n,s} \s$.

The maximal polynomial ideals  associated  to the natural norms are easily obtain using Proposition \ref{descripcion} and the fact that $\Q_{/ \beta \s}$ and $\Q_{\s \gamma /}$ are associated to the norms $\s \beta' /$ and $/ \gamma' \s$ respectively.
\end{proof}

\bigskip

The 2-fold tensor norms $\pi_2$ and $\s \varepsilon_2 /$ (which is equivalent to $w_2'$) share two interesting properties. The first property is that they dominate their dual tensor norm. Clearly $\pi_2'=\varepsilon_2\le \pi_2 $. Also, it can be seen in \cite[27.2]{DefFlo93} that  $w_2$ is dominated by $w_2'$ (or, analogously, $/\pi_2\s$ is dominated by $\s \varepsilon_2 /$). The second property is that both $\pi_2$ and $w_2'$ preserve the Banach algebra structure \cite{Carne78}.
These two properties are enjoyed, of course, by their corresponding 2-fold s-tensor norms (see the proof of Theorem~\ref{son seis} for the first one, and Section~\ref{section preserving} for the second one).
As we have already seen, the $n$ dimensional analogue of the s-tensor norm $\s \varepsilon_{2,s} /$ splits into two non-equivalent ones when passing from tensor products of order 2 to tensor products of order $n\geq3$. Namely,  $\s \varepsilon_{n,s} /$ and $ \s / \pi_{n,s}  \s /$.
It is remarkable that the two mentioned properties are  enjoyed only by $\s / \pi_{n,s}  \s /$ and not by $\s \varepsilon_{n,s} /$, as seen in Theorems~\ref{son seis} and~\ref{cuales preservan algebras}. Therefore, we could say that, in some sense, the $n$-fold symmetric analogue of $w_2'$ for $n\ge 3$ should be  $\s / \pi_{2,s}  \s /$ rather than the simpler (and probably nicer) $\s \varepsilon_{2,s} /$.

In the proof of Theorem~\ref{son seis} we have shown that $\s \varepsilon_{n,s} /$ and $\s / \pi_{n,s} \s /$ are not equivalent on $\otimes^{n,s}\ell_1$. One may wonder if there exist an infinite dimensional Banach space such that  $\s \varepsilon_{n,s} /$ and $\s / \pi_{n,s} \s /$  are equivalent in $\otimes^{n,s}E$ for $n\geq 3$. We see that this is not the case in Theorem~\ref{no son equivalentes}.
To prove the theorem we will need the following  proposition.

\begin{proposition}\label{accesibles los asociados a inyectivas}
Let $\q$ be a polynomial ideal and $\beta$ its associated tensor norm. If $\beta$ is injective then $\q$ is accessible.
\end{proposition}

\begin{proof}
Let $q$ be a finite type polynomial on $E$ and choose $(x_j')_{j=1}^r$ in $E'$ such that $q= \sum_{j=1}^r (x_j')^n$. We set $L=\bigcap_{j=1}^r Ker(x_j')$, which is a finite codimensional subspace of $E$.
For each $j = 1, \dots, r$, let $\overline{x'}_j \in (E/L)'$ be defined by $ \overline{x'}_j (\overline{x}) := x_j'(x)$ (where $\overline x$ denotes the class of $x$ in $E/L$).
If $Q_E^L: E\to  E/L$ is the quotient map and $p$ is the polynomial on $E/L$ given by $p= \sum_{j=1}^r (\overline{x'}_j)^n$, we have $q = p \circ Q_{L}^E$.
Also, since $\beta$ is injective we have the isometry $$ \otimes^{n,s} (Q_{L}^E)': \otimes^{n,s}_{\beta} (E/L)'  \overset 1 \hookrightarrow \otimes^{n,s}_{\beta} E'.$$
This altogether gives
\begin{align*}
\|p\|_{\q} & = \beta \big(\sum_{j=1}^r \otimes^n \overline{x}_j',\otimes^{n,s} (E/L)' \big) \\
& = \beta \big(\otimes^{n,s} (Q_{L}^E)'(\sum_{j=1}^r \otimes^n \overline{x'}_j),\otimes^{n,s} E'\big) \\
& = \beta \big(\sum_{j=1}^r \otimes^n x_j',\otimes^{n,s} E'\big) = \|q\|_{\q},
\end{align*}
which shows the accessibility of $\Q$.
\end{proof}

\begin{theorem}\label{no son equivalentes}
For $n \geq 3$, $\s \varepsilon_{n,s} /$ and $\s / \pi_{n,s} \s /$ are equivalent in $\otimes^{n,s}E$ if and only if $E$ is finite dimensional. The same happens if   $/ \pi_{n,s} \s$ and $/ \s \varepsilon_{n,s} / \s$ are equivalent on $E$.
\end{theorem}

\begin{proof}
We will first prove that if $E$ is infinite dimensional, then $/ \pi_{n,s} \s$ and $/ \s \varepsilon_{n,s} / \s$ are not equivalent in $\otimes^{n,s}E$.
Suppose they are. Then, if we denote by $\p_e$ the ideal of extendible polynomials,  we have
$$ \p_e(E) = \big( \os_{/ \pi_{n,s} \s} E \big)' = \big( \os_{/
\s \varepsilon_{n,s} / \s} E \big)'= \Q_{/ \s \varepsilon_{n,s} / \s}(
E).$$
By the open mapping theorem, there must be a constant $M> 0$ such that $\|p\|_{\Q_{/ \s \varepsilon_{n,s} / \s}(E)}  \leq M \|p\|_{\p_e(
E)},$ for every
extendible polynomial $p$ on
$E$.
If $F$ is a subspace of  $E$, any extendible polynomial on
$F$ extends to an extendible polynomial on $E$ with the same extendible norm.
Therefore, for every subspace $F$ of $E$  and every extendible
polynomial $q$ on $F$, we have $ \|q\|_{\Q_{/ \s \varepsilon_{n,s} /
\s}(F)} \leq M \|q\|_{\p_e(F)}$.

Since $E$ is an infinite dimensional space, by Dvoretzky's theorem it contains $(\ell_2^k)_k$ uniformly. Then  there exists a constant $C> 0$ such that for every $k$ and every polynomial $q$ on $\ell_2^k$, we have
$$ \|q\|_{\Q_{/ \s \varepsilon_{n,s} /
\s}(\ell_2^k)} \leq C \|q\|_{\p_e(\ell_2^k)}.$$ Since the ideal of extendible polynomials is maximal (it is dual to an s-tensor norms), we deduce that  \begin{equation}\label{extendibles incluidos}\p_e(\ell_2) \subset \Q_{/ \s \varepsilon_{n,s} / \s}(\ell_2).\end{equation}
Let us show that this is not true. Since $/ \s \varepsilon_{n,s} / \s$ is injective and we have an inclusion $\ell_2 \hookrightarrow L_1[0,1]$, each $p \in  \Q_{/ \s \varepsilon_{n,s} / \s}(\ell_2)$ can be extended to a $/ \s \varepsilon_{n,s} / \s$-continuous  polynomial $\widetilde{p}$  on $L_1[0,1].$ Now, $\varepsilon_{n,s}$ coincides with $\s \varepsilon_{n,s} /$ on $L_1[0,1]$, which is in turn dominated by $/ \s \varepsilon_{n,s} / \s$. Therefore, the polynomial $\widetilde{p}$ is actually $\varepsilon_{n,s}$-continuous or, in other words, integral. Since $\widetilde{p}$ extends $p$, this must also be integral, and we have shown that $\Q_{/ \s \varepsilon_{n,s} / \s}(\ell_2)$ is contained in $\p_I(\ell_2)$.
But it is shown in  \cite{CarDim07,Per04(4-linear),Varapoulus(1975)} that there are always extendible non-integral polynomials on any infinite dimensional Banach space, so (\ref{extendibles incluidos}) cannot hold.
This contradiction shows that $/ \pi_{n,s} \s$ and $/ \s \varepsilon_{n,s} / \s$ cannot be equivalent on $E$.

Now we will show that $\s \varepsilon_{n,s} /$ and $\s / \pi_{n,s} \s /$ are not equivalent in $\otimes^{n,s}E$, for any infinite dimensional Banach space $E$.
Suppose they are. By duality, we have $ \q_{\s \varepsilon_{n,s} /}= \q_{\s / \pi_{n,s} \s /}$ with equivalent norms.
Proposition \ref{accesibles los asociados a inyectivas} ensures that the polynomial ideals $ \q_{\s \varepsilon_{n,s} /},\; \q_{\s / \pi_{n,s} \s /}$ are both accesible, since they are associated to the injective norms  $/ \pi_{n,s} \s, $ and $ / \s \varepsilon_{n,s} / \s$ respectively. Thus, by \cite[Proposition 3.6]{Flo01} we have:
$$ \widetilde{\otimes}^{n,s}_{/ \pi_{n,s} \s}  E' \overset{1}{\hookrightarrow} \q_{\s \varepsilon_{n,s} /}(E), \text{ and } \widetilde{\otimes}^{n,s}_{/ \s \varepsilon_{n,s} / \s}  E' \overset{1}{\hookrightarrow} \q_{\s / \pi_{n,s} \s /}(E).$$
But this implies that $/ \pi_{n,s} \s$ and $/ \s \varepsilon_{n,s} / \s$ are equivalent in $\os E'$, which is impossible by the already proved first statement of the Theorem.
\end{proof}

\section{s-Tensor norms preserving Banach algebra structures}\label{section preserving}
Carne in \cite{Carne78} described the natural 2-fold tensor norms that preserve Banach algebras.
In this section we will show that $\pi_{n,s}$ and $\s / \pi_{n,s} \s /$ are the only  natural s-tensor norms that preserve the algebra structure.

For a given Banach algebra A we will denote $m(A) : A \otimes_{\pi_2} A \to A$ the map induced by the multiplication $A \times A \to A$.
The following theorem is a symmetric version of Carne \cite[Theorem~1]{Carne78}. Its proof is obtained by adapting the one in \cite{Carne78} for the symmetric setting.

\begin{theorem}\label{carne}
For an s-tensor norm $\beta$ of order $n$ the following conditions are equivalent:
\begin{enumerate}
\item If $A$ is Banach algebra, the $n$-fold symmetric tensor product $\widetilde{\otimes}^{n,s}_\beta A$ is a Banach algebra with the natural algebra structure.
\item For all Banach spaces $E$ and $F$ there is a natural continuous linear map
$$f: \big( \os_\beta E \big) \otimes_{\pi_2} \big( \os_\beta F \big) \to \big( \os_\beta (E \otimes_{\pi_2} F) \big)$$ with
$$f \big( (\otimes^n x)  \otimes (\otimes^n y) \big) = \otimes^n (x \otimes y).$$
\item For all Banach spaces $E$ and $F$ there is a natural continuous map
$$g: \big( \os_{\beta'} (E \otimes_{\varepsilon_2} F)  \big) \to ( \os_{\beta'} E) \otimes_{\varepsilon_2}  ( \os_{\beta'} F)$$ given by
$$ g \big( \otimes^n (x \otimes y) \big) = (\otimes^n x) \otimes (\otimes^n y).$$
\item For all Banach spaces $E$ and $F$ there is a natural continuous map
$$h: \os_{\beta'} \mathcal{L}(E,F) \to \mathcal{L}( \os_{\beta} E, \os_{\beta'} F),$$
with $$h( \otimes^n T) (\otimes^n x) = \otimes^n (Tx).$$
\end{enumerate}
If one, hence all, of the above hold, then there are constants $c_1, c_2, c_3, c_4$ so that
\begin{enumerate}
\item $\| m ( \widetilde{\otimes}^{n,s}_\beta A )\| \leq c_1 \|m(A)\|^n.$
\item $\|f\| \leq c_2$ for all $E$ and $F$.
\item $\|g\| \leq c_3$ for all $E$ and $F$.
\item $\|h\| \leq c_4$ for all $E$ and $F$.
\end{enumerate}
and the least values of these four agree.

\end{theorem}

If the s-tensor norm $\beta$ preserves Banach algebras, then we will call the common least value of the constants in the theorem, \emph{the Banach algebra constant} of $\beta$.

An important comment is in order: if  we take $E=F$ and $T=id_E$ in $(4)$, then we obtain
$\|h( \os id_E )\|~\leq~c_4$. But it is plain that $h(\otimes^n id_E)$ is just $id_{\os E}$. Therefore, we have
$$ \| id_{\os E}: \os_{\beta} E \to \os_{\beta'} E \| \leq c_4,$$
which means that $\beta' \leq c_4 \beta$. So we can state the following remark.

\begin{remark}
If $\beta$ is an s-tensor norm which preserves Banach algebras there is a constant $k$ such that $\beta' \leq k \beta$.
\end{remark}

The following Theorem is the main result of this section. The proof that $\pi_s$ preserves Banach algebra is similar to one for $\pi_2$ in \cite{Carne78}, and we include it for completeness.

\begin{theorem}\label{cuales preservan algebras} The only natural s-tensor norms of order $n$ which preserves Banach algebras are: $\pi_{n,s}$ and $\s / \pi_{n,s} \s /$. Furthermore, the Banach algebra constants of both norm are exactly one.
\end{theorem}

\begin{proof} It follows from Theorem~\ref{son seis} and the previous remark that  $\pi_{n,s}$ and $\s / \pi_{n,s} \s /$ are the only candidates among natural s-tensor norms to preserve  Banach algebras.

First we prove that $\pi_s$ preserves Banach algebra. By Theorem ~\ref{carne}, it is enough to show, for any  pair of Banach spaces $E$ and $F$, that the mapping
$$f: \big( \os_{\pi_{n,s}} E \big) \otimes_{\pi_2} \big( \os_{\pi_{n,s}} F \big) \to \big( \os_{\pi_{n,s}} (E \otimes_{\pi_2} F) \big)$$ defined by
$$f \big( (\otimes^n x)  \otimes (\otimes^n y) \big) = \otimes^n (x \otimes y),$$ has norm less or equal than 1. Fix $\varepsilon >0$. Given $w \in \big( \os E \big) \otimes \big( \os F \big)$, we can write it as
$$w = \sum_{ i = 1}^r u_i \otimes v_i,$$ with $$\sum_{i = 1}^r \pi_{n,s}(u_i) \pi_{n,s}(v_i) \leq \pi_2(w) (1 + \varepsilon)^{1/3}.$$
Also, for each $i=1,\dots, r$ we write $u_i$ and $v_i$ as
$$u_i= \sum_{j=1}^{J(i)} \otimes^n x_j^i \in \os E,\quad v_i= \sum_{k=1}^{K(i)} \otimes^n y_k^i \in \os F,$$
with $$ \sum_{j=1}^{J(i)}  \|x_j^i\|^n \leq \pi_{n,s}(u_i) (1 + \varepsilon)^{1/3}, \quad  \sum_{k=1}^{K(i)} \|y_k^i\|^n \leq \pi_{n,s}(v_i) (1+\varepsilon)^{1/3}.$$
We  have
$$f(w) = \sum_{i=1}^r \sum_{\substack
{1\le j \leq J(i)\\ 1\le k \leq K(i)}} \otimes^n ( x_j^i \otimes y_k^i),$$
and then
\begin{align*}
\pi_{n,s}(f(w)) & \leq \sum_{i=1}^r \sum_{\substack
{1\le j \leq J(i)\\ 1\le k \leq K(i)}} \pi_2 (x_j^i \otimes y_k^i)^n \\
& = \sum_{i=1}^r \sum_{\substack
{1\le j \leq J(i)\\ 1\le k \leq K(i)}}  \|x_j^i\|^n \|y_k^i\|^n \\
& = \sum_{i=1}^r \big(\sum_{j \leq J(i)} \|x_j^i\|^n \big) \big( \sum_{k \leq K(i)} \|y_k^i\|^n \big) \\
& = \sum_{i=1}^r \pi_{n,s}(u_i) (1 + \varepsilon)^{1/3} \pi_{n,s}(v_i) (1+\varepsilon)^{1/3} \\
& = (1+\varepsilon)^{2/3} \sum_{ i = 1}^r \pi_2(u_i) \pi_2(v_i) \leq (1+ \varepsilon) \pi(w).
\end{align*}
From this we conclude that $\|f\| \leq 1$.

\bigskip
To prove that $\s / \pi_{n,s} \s /$ preserves Banach algebras we  need two technical lemmas.

\begin{lemma} \label{lemma1}
Let $Y$ and $Z$ be Banach spaces. The operator
$$ \phi: \os_{/ \pi_{n,s} \s} \L(\ell_1(B_Y),Z) \to \L \big( \os_{/ \pi_{n,s} \s} \ell_1(B_Y), \os_{/ \pi_{n,s} \s} Z \big)$$
given by $$ \phi(\otimes^n T )( \otimes^n u)= \otimes^n Tu,$$
has norm less or equal than 1.
\end{lemma}

\begin{proof}
The mapping
\begin{eqnarray*} \L \big( \ell_1(B_Y), \ell_\infty(B_{Z'}) \big) & \to & \L \big( \os_{/ \pi_{n,s} \s} \ell_1(B_Y), \os_{/ \pi_{n,s} \s} Z \big)\\
T &\mapsto &\otimes^n T\end{eqnarray*}
is an $n$-homogeneous polynomial, which has norm one by the metric mapping property of the norm $ / \pi_{n,s} \s$. As a consequence, its linearization is a norm one operator from
$ \os_{/ \pi_{n,s} \s} \L \big( \ell_1(B_Y), \ell_\infty(B_{Z'}) \big)$ to
$\L \big( \os_{/ \pi_{n,s} \s} \ell_1(B_Y), \os_{/ \pi_{n,s} \s} Z \big)$.
Since $\L \big( \ell_1(B_Y), \ell_\infty(B_{Z'}) \big)$ is an $\L_{\infty}$ space we have
$$ \os_{/ \pi_{n,s} \s} \L \big( \ell_1(B_Y), \ell_\infty(B_{Z'}) \big) \overset 1 = \os_{ \pi_{n,s} } \L \big( \ell_1(B_Y), \ell_\infty(B_{Z'}) \big).$$
This shows that the canonical mapping
\begin{equation*}
\xymatrix{ \os_{/ \pi_{n,s} \s} \L \big( \ell_1(B_Y), \ell_\infty(B_{Z'}) \big)  \ar[r]  & \L \big( \os_{/ \pi_{n,s} \s} \ell_1(B_Y), \os_{ / \pi_{n,s} \s} \ell_\infty(B_{Z'}) \big) }
\end{equation*}
has norm 1.

On the other hand, the following diagram commutes
\begin{equation*}
\xymatrix{
\os_{/ \pi_{n,s} \s} \L \big( \ell_1(B_Y), \ell_\infty(B_{Z'}) \big)  \ar[rr] & & \L \big( \os_{/ \pi_{n,s} \s} \ell_1(B_Y), \os_{ / \pi_{n,s} \s} \ell_\infty(B_{Z'}) \big)  \\
\os_{ / \pi_{n,s} \s} \L(\ell_1(B_Y), Z) \ar[rr]^{\phi} \ar@{^{(}->}[u] & & \L( \os_{/ \pi_{n,s} \s} \ell_1(B_Y), \os_{ / \pi_{n,s} \s} Z )  \ar@{^{(}->}[u] }.
\end{equation*}
Here the vertical arrows are the natural inclusion, which  are actually isometries since the norm $ / \pi_{n,s} \s$ is injective. The horizontal arrow above is the canonical mappings whose norm was shown to be one. Therefore, the norm of $\phi$ must be less or equal to one.
\end{proof}

Before we state our next lemma, we observe that linear operators from $X_1$ to $\L(X_2,X_3)$ identify (isometrically) with bilinear operators  from $X_1\times X_2$ to $X_3$ and, consequently, with linear operators from $X_1\otimes_\pi X_2$ to $X_3$. The isometry is given by \begin{eqnarray}
\L(X_1,\L(X_2,X_3)) &\to &\L(X_1\otimes_\pi X_2,X_3) \nonumber \\
T&\mapsto& B_T,\label{identificacion}
\end{eqnarray}
where $B_T(x_1\otimes x_2)= T(x_1)(x_2)$.

\begin{lemma}
Let $E$ and $F$ be Banach spaces. The operator
$$ \rho: \big( \os_{/ \pi_{n,s} \s} \ell_1(B_E) \big) \otimes_{\pi_2} \big( \os_{/ \pi_{n,s} \s} \ell_1(B_F) \big) \to \os_{/ \pi_{n,s} \s} \big( \ell_1(B_E) \otimes_{\pi_2} \ell_1(B_F) \big)$$
given by
$$\rho\big(( \otimes^n u ) \otimes ( \otimes^n v )\big) = \otimes^n ( u \otimes v),$$
has norm less or equal than 1.
\end{lemma}

\begin{proof}
If we take $Y=F$ and $Z=\ell_1(B_E) \otimes_{\pi_2} \ell_1(B_F)$ in Lemma~\ref{lemma1},  we see that
the operator
$$ \phi: \os_{/ \pi_{n,s} \s} \L(\ell_1(B_F),\ell_1(B_E) \otimes_{\pi_2} \ell_1(B_F)) \to \L \big( \os_{/ \pi_{n,s} \s} \ell_1(B_E), \os_{/ \pi_{n,s} \s} (\ell_1(B_E) \otimes_{\pi_2} \ell_1(B_F)) \big)$$ has norm at most $1$.
Also the application $J: \ell_1(B_E) \to \L \big( \ell_1(B_F), \ell_1(B_E) \otimes_{\pi_2} \ell_1(B_F) \big)$ defined by $Jz(w)~=~z~\otimes~w$ has norm 1.
Therefore, the norm of the map  $\psi := \phi \circ \otimes^{n,s} J$ between the corresponding  $ / \pi_{n,s} \s  $-tensor products is at most one.

Now, with the identification given in (\ref{identificacion}), the operator $ \rho $ is precisely $B_\psi$, and since (\ref{identificacion}), we conclude that $\rho$ has norm at most one.
\end{proof}

Now we are ready to prove that $ \s / \pi_{n,s} \s /$ preserves Banach algebras with Banach algebra constant 1.
Again by Theorem~\ref{carne}, it is enough to show that, for Banach spaces $E$ and $F$, the map
$$f: \big( \os_{\s / \pi_{n,s} \s /} E \big) \otimes_{\pi_2} \big( \os_{\s / \pi_{n,s} \s /} F \big) \to  \os_{\s / \pi_{n,s} \s /} (E \otimes_{\pi_2} F)$$ defined by
$$f \big( (\otimes^n x)  \otimes (\otimes^n y) \big) = \otimes^n (x \otimes y),$$ has norm at most one.
The following diagram, where the vertical arrows are the canonical quotient maps, commutes:
\begin{equation*}
\xymatrix{  \big( \os_{/ \pi_{n,s} \s} \ell_1(B_E) \big) \otimes_{\pi_2} \big( \os_{/ \pi_{n,s} \s} \ell_1(B_F) \big) \ar[rr]^{\rho} \ar@{->>}[d]  & & \big( \os_{ / \pi_{n,s} \s } (\ell_1(B_E) \otimes_{\pi_2} \ell_1(B_F))  \big) \ar@{->>}[d] \\
\big( \os_{\s / \pi_{n,s} \s /} E  \big) \otimes_{\pi_2} \big( \os_{\s / \pi_{n,s} \s /} F  \big) \ar[rr]^{f} & & \big( \os_{\s / \pi_{n,s} \s /} (E \otimes_{\pi_2} F)  \big) }.
\end{equation*}
By the previous Lemma,  $\rho$ has norm less than or equal to one, and so is the norm of $f$, since the other mappings are quotients.
\end{proof}

\end{document}